\documentclass[a4paper,10pt,twoside]{article}
\usepackage[active]{srcltx}
\usepackage{amsmath,amssymb,epsf}
\usepackage{times}
\usepackage{mathrsfs}
\usepackage{latexsym,amsopn}
\usepackage{amsfonts,amsbsy,amscd,stmaryrd}
\usepackage{theorem}
\usepackage{paralist}
\usepackage{geometry}
\numberwithin{equation}{section}

\renewcommand\Im{{\rm Im\,}}
\renewcommand\Re{{\rm Re\,}}

\newcommand\Pv{{\rm Pv}}

\newcommand{\cc}{{\cal C}}

\newcommand{\e}{{\rm e}}

\renewcommand{\i}{{\rm i}}
\renewcommand{\d}{{\rm d}}

\renewcommand{\sp}{{\rm sp\,}}

\newcommand{\grintl}{[\kern-.18em [}
\newcommand{\grintr}{]\kern-.18em ]}

\newcounter{resultcounter}[section]

\newtheorem{theorem}[resultcounter]{Theorem}
\newtheorem{lemma}[resultcounter]{Lemma}
\newtheorem{proposition}[resultcounter]{Proposition}
\newtheorem{corollary}[resultcounter]{Corollary}
\newtheorem{definition}[resultcounter]{Definition}
\newtheorem{remark}[resultcounter]{Remark}

\newcommand\bep{\begin{proposition}}
\newcommand\eep{\end{proposition}}
\newcommand\ber{\begin{remark}}
\newcommand\eer{\end{remark}}
\newcommand\bel{\begin{lemma}}
\newcommand\eel{\end{lemma}}
\newcommand\bet{\begin{theorem}}
\newcommand\eet{\end{theorem}}
\newcommand\bex{\begin{example}}
\newcommand\eex{\end{example}}
\newcommand\bed{\begin{definition}}
\newcommand\eed{\end{definition}}
\newcommand\bea{\begin{assumption}}
\newcommand\eea{\end{assumption}}
\newcommand\bec{\begin{corollary}}
\newcommand\eec{\end{corollary}}
\newcommand{\beq}{\begin{equation}}
\newcommand{\eeq}{\end{equation}}
\newcommand{\beqa}{\begin{eqnarray}}
\newcommand{\eeqa}{\end{eqnarray}}

\def\one{{\mathchoice {\rm 1\mskip-4mu l} {\rm 1\mskip-4mu l} {\rm
      1\mskip-4.5mu l} {\rm 1\mskip-5mu l}}} 

  \def\cC{{\cal C}}
  \def\cF{{\cal F}}
 \def\cH{{\cal H}}

\newcommand{\R}{{\mathbb R}}
\newcommand{\N}{{\mathbb N}}

\newcommand{\C}{{\mathbb C}}
\newcommand{\Z}{{\mathbb Z}}

\def\proof{\noindent{\bf Proof.}\ \ }
\def\qed{\hfill $\Box$\medskip}

\newenvironment{innerlist}[1][\enskip\textbullet]%
        {\begin{compactitem}[#1]}{\end{compactitem}}

\title{Weber-Schafheitlin integrals with arbitrary exponent}

\author{Micha\l\ Wrochna}

\begin{document}

\maketitle

\begin{abstract}
We present explicit formulae for Weber-Schafheitlin type integrals and give them an interpretation as the kernel of a physically relevant operator related to the hamiltonian of Aharanov and Bohm. In particular, we derive explicit formulae for Weber-Schafheitlin type integrals with exponent larger or equal 1, which are distributions on $\R_+$. We discuss several special cases.
\end{abstract}

\footnotetext{\small {\bf \ Acknowledgements} The author would like to thank especially J.~Derezi\'{n}ski for suggesting the topic of this paper and proofreading, and V.~Georgescu and S.~Richard for useful advice. The author is also grateful to the Mathematics Department of Cergy-Pontoise University for hospitality and to the Government of France for financial support. A large part of this paper was written at the Department of Mathematical Physics, University of Warsaw.}


\section{Introduction}\label{s:intro}

Our aim is to calculate the integral 
\begin{equation} \label{eq:besseli}
\int_0^\infty \kappa^{\rho}J_{\mu}(x\kappa)J_{\nu}(\kappa)\d\kappa ,
\end{equation}
for suitable $\mu,\nu \in \C$ and $x\in \R_+ :=(0,\infty)$, where $J_{\mu}$ is the Bessel function of the first kind of order $\mu$. In the literature, it is known as the Weber-Schafheitlin discontinuous integral with exponent $\rho$ (or $-\rho$, depending on the convention).

In the case when $\Re\rho<1$ (and under some additional assumptions on $\mu,\nu$) it is convergent to a function, in general not continuous in $x=1$. It has been derived in several ways and analysed in many special cases, for which we refer in particular to \cite{W} and \cite{DF}. It has been applied in numerous problems, let us mention here only two recent works --- \cite{HN}, \cite{SS}. 

The case $\Re\rho\geq1$ is more delicate and requires a distributional approach. Nevertheless, it is quite natural to consider it; indeed, it appears in some problems where it plays an important role (\cite{KR2}, \cite{KeR}). In addition to that, we show in this paper that it is the kernel of an operator physically relevant for the Aharanov-Bohm system. There have already been successful attempts to derive useful expressions for the distributional case of (\ref{eq:besseli}) for special values of parameters, by Kellendonk and Richard \cite{KR} (for $\rho=1$), by Miroshin \cite{M} ($\rho\to1$ asymptotic) and by Salamon and Walter \cite{SW} (recurrence formulae and special values of parameters with $\rho\in\Z$). Motivated by the wish of exhausting all unsolved cases, we provide explicit formulae for (\ref{eq:besseli}) for arbitrary $\rho$ with positive real part. We discuss also some special cases and compare them with the results mentioned earlier.

The paper is constructed as follows. Section \ref{motywacja} serves as an additional motivation, linking the hamiltonian of Aharanov and Bohm to the integral we discuss further. In Section \ref{watsonowy}, we derive formulae for integrals involving the modified Bessel function of the second type $K_{\mu}$, as in the approach of Dixon and Ferrar \cite{DF}. We then use them to compute the integrals involving the Hankel functions of the first and second kind --- $H^+_{\mu}$, $H^-_{\mu}$,
\begin{equation} \label{eq:hankeli}
\int_0^\infty \kappa^{\rho}H^{\pm}_{\mu}(x\kappa)J_{\nu}(\kappa)\d\kappa ,
\end{equation}
closely related to (\ref{eq:besseli}). Our treatment of the integrals (\ref{eq:hankeli}) is a generalization of the approach adopted in \cite{KR}. We consider separately the cases $\Re\rho\leq0$ and $\Re\rho>0$. The second one is examinated in Section 4 and includes in particular the distributional cases. The main result for the integral (\ref{eq:besseli}) is contained in Proposition \ref{wynik2}, and follows as a simple consequence of the computations of (\ref{eq:hankeli}).

\subsection{The main result}

In the following proposition, we give formulae for the Weber-Schaftheitlin integral (\ref{eq:besseli}) with exponent $\Re\rho>0$. The result is a distribution on $\R_+$ (in general not regular), and its form depends on whether $\rho$ is an integer number (which is not the case if $\Re\rho\leq0$). We use the notation for the rescaled Gauss hypergeometric function 
\[ _{2}F_1^{\rm I}\left(a,b;c;z\right):=\frac{1}{\Gamma(c+1)} \ _{2}F_1\left(a,b;c;z\right).\]

\begin{proposition}\label{wynik2}
For any $\mu,\nu\in\C$ and $\Re\rho>0$ satisfying $\Re(\rho+\nu+1)>\left|\mu\right|$, and $x\in\R_+$, the integral $\int_0^\infty \kappa^{\rho}J_{\mu}(x\kappa)J_{\nu}(\kappa)\d\kappa$ (\ref{eq:besseli}) equals:

\begin{innerlist}
\item for $\rho\notin\Z$,
\begin{eqnarray*}
\frac{2^{\rho}}{\sin\pi\rho}\bigg[\left\{(x-1)_-^{-\rho}\sin\left({\textstyle \pi \frac{1-\rho-\mu+\nu}{2}}\right)+(x-1)_+^{-\rho}\sin\left({\textstyle \pi \frac{1+\rho-\mu+\nu}{2}}\right)\right\} \frac{x^{-1+\rho-\nu}}{(1+x)^{\rho}} \ _{2}F_1^{\rm I}\left({\textstyle \frac{1-\rho+\mu+\nu}{2}},{\textstyle\frac{1-\rho-\mu+\nu}{2}};-\rho+1;1-x^{-2}\right) \\  - \sin\left({\textstyle \pi \frac{1+\rho-\mu+\nu}{2}}\right) \frac{\Gamma({\textstyle \frac{1+\rho+\mu+\nu}{2}})\Gamma({\textstyle \frac{1+\rho-\mu+\nu}{2}})}{\Gamma({\textstyle \frac{1-\rho+\mu+\nu}{2}})\Gamma({\textstyle \frac{1-\rho-\mu+\nu}{2}})}x^{-1-\rho-\nu} \ _{2}F_1^{\rm I}\left({\textstyle \frac{1+\rho+\mu+\nu}{2}},{\textstyle\frac{1+\rho-\mu+\nu}{2}};\rho+1;1-x^{-2}\right) \bigg]
\end{eqnarray*} 

\item for  $\rho\in\Z$ and $\frac{1-\rho\pm\mu+\nu}{2}\notin\Z$,
\begin{eqnarray*}
\frac{2^{\rho}}{\pi} \bigg[ \left\{ \sin\left({\textstyle \pi \frac{1+\rho-\mu+\nu}{2}}\right)\left[(x-1)^{-\rho}_- + (-1)^{\rho}(x-1)^{-\rho}_+\right]+\cos\left({\textstyle\pi \frac{1+\rho+\mu-\nu}{2} }\right) (-1)^{\rho} \ \pi\frac{\delta^{(\rho-1)}(1-x)}{(\rho-1)!} \right\} \\ \times \ x^{-1+\rho-\nu}(1+x)^{-\rho} S_{\mu,\nu,\rho}(1-x^{-2}) +(-1)^{\rho} \frac{\Gamma({\textstyle \frac{1+\rho+\mu+\nu}{2}})\Gamma({\textstyle \frac{1+\rho-\mu+\nu}{2}})}{\Gamma({\textstyle \frac{1-\rho+\mu+\nu}{2}})\Gamma({\textstyle \frac{1-\rho-\mu+\nu}{2}})}x^{-1-\rho-\nu} \bigg\{ \sin\left({\textstyle \pi \frac{1+\rho-\mu+\nu}{2}}\right) T_{\mu,\nu,\rho}(1-x^{-2}) \\ - \ \left[\sin\left({\textstyle \pi \frac{1+\rho-\mu+\nu}{2}}\right)\log x^{-2}(x+1)\left|x-1\right|+\cos\left({\textstyle \pi \frac{1+\rho-\mu+\nu}{2}}\right)\pi\theta(x-1)\right] \ _{2}F_1^{\rm I}\left({\textstyle \frac{1+\rho+\mu+\nu}{2}},{\textstyle\frac{1+\rho-\mu+\nu}{2}};\rho+1;1-x^{-2}\right)\bigg\} \bigg].
\end{eqnarray*}
\end{innerlist}
where all the equalities hold in the sense of distributions on $\R_+$. In both cases, the RHS is the sum of a distribution multiplied by a smooth function (the first term) and of a locally integrable function (the second term). We have denoted $\theta(x)$ the Heaviside theta function, and the functions $S_{\mu,\nu,\rho}, T_{\mu,\nu,\rho}$ are defined for $\left|z\right|<1$ by
\[
S_{\mu,\nu,\rho}(z):=\sum_{k=0}^{\rho-1}\frac{\left({\textstyle \frac{1-\rho+\mu+\nu}{2}}\right)_k \left({\textstyle \frac{1-\rho-\mu+\nu}{2}}\right)_k}{(1-\rho)_k k!}z^k,
\]
\begin{eqnarray*}
T_{\mu,\nu,\rho}(z):=\sum_{k=0}^{\infty}\frac{\left({\textstyle \frac{1+\rho+\mu+\nu}{2}}\right)_k \left({\textstyle \frac{1+\rho-\mu+\nu}{2}}\right)_k}{(\rho+k)!k!}z^k \\ \times \ \left\{\psi(k+1)+\psi(\rho+k+1)-\psi({\textstyle \frac{1+\rho+\mu+\nu}{2}}+k)-\psi({\textstyle \frac{1+\rho-\mu+\nu}{2}}+k)\right\},
\end{eqnarray*}
with $\psi(y):=\frac{\d}{\d y}\Gamma(y)/\Gamma(y)$.
We use also the definition, for $\Re\lambda>-1$:
\[
(x-1)^{\lambda}_{-}:=\begin{cases}|x-1|^{\lambda},& x<1\\ 0, & x\geq1 \end{cases},
\]
\[
(x-1)^{\lambda}_{+}:=\begin{cases}0,& x<1\\ (x-1)^{\lambda}, & x\geq1 \end{cases},
\]
and extend it analytically in the sense of distributions to all values of $\lambda\in\C$ (see Remark \ref{remark1}).

\end{proposition}

\section{The Weber-Schafheitlin integral as an integral kernel}\label{motywacja}

In the following section, we quote some results obtained in \cite{BDG} for the Aharonov-Bohm hamiltonian and conclude from them, that the Weber-Schafheitlin integral describes the integral kernel of a physically relevant operator. We motivate thus the need for explicit formulae, valid in particular for the distributional $\Re\rho\geq1$ case.

We consider the Hilbert space $L^2(\R^2)$ and denote its inner product $(\cdot|\cdot)$. Since it will be convenient to use polar coordinates $r, \phi$ on $\R^2$, we introduce the unitary transformation \[L^2(\R^2)\ni f\mapsto Uf\in L^2(0,\infty)\otimes L^2(-\pi,\pi)\] given by $Uf(r,\phi)=\sqrt{r}f(r\cos\phi,r\sin\phi)$, which allows us to identify $L^2(\R^2)$ with
$L^2(0,\infty)\otimes L^2(-\pi,\pi)$.

The Aharanov-Bohm hamiltonian in polar coordinates is
\[
H^{{\rm AB}}_\lambda:=-\partial_r^2-\frac{1}{r^2}(\partial_\phi+\i\lambda)^2,
\]
understood as the self-adjoint operator associated to the differential expression above (defined on an appropriate domain). We allow for the moment the parameter $\lambda$ to be any complex number. 

Since the self-adjoint operator $L:=-\i\partial_{\phi}$ has spectrum
$\sp(L)=\Z$, we have the decomposition $L^2(\R^2)
=\mathop{\oplus}\limits_{k\in\Z}\cH_k$ where $\cH_k$ is the spectral
subspace of $L$ for the eigenvalue $k$. With the help of $U$ we can
identify $\cH_k$ with $L^2(\R)$. Since $L$ commutes with $H^{{\rm AB}}_\lambda$, we obtain the decomposition
\[
U H^{{\rm AB}}_\lambda U^*=\mathop{\oplus}_{k\in\Z}H_{k+\lambda},
\]
where $H_{\mu}$ acts as the differential operator $-\partial^2_x+\frac{\mu^2-\frac{1}{4}}{x^2}$, when restricted to $\cC_c(\R_+)$.


We now assume $\mu>-1$. We gather some results from \cite{BDG} about the operator $H_{\mu}$. We will need first to define the following symmetric operator, corresponding up to a constant factor to the so-called Hankel transformation:
\begin{definition}
$\cF_{\mu}$ is the operator on $L^2(0,\infty)$ given by
\[
\left(\cF_{\mu}f\right)(k):=\int_0^\infty J_{\mu}(kx)\sqrt{kx}f(x)\d x
\]
\end{definition}
We have then:
\begin{theorem}
Let $0<a<b<\infty$ and denote $\one_{[a,b]}$ the characteristic function of the corresponding interval. The integral kernel of $\one_{[a,b]}(H_{\mu})$ is
\[
\one_{[a,b]}(H_{\mu})(x,y) = \int_{\sqrt a}^{\sqrt b} \sqrt{xy}J_{\mu}(x \kappa) J_{\mu}( y\kappa)\kappa\d \kappa,
\]
considered as a quadratic form on $\cC_c^{\infty}(\R_+)$, that is, explicitly:
\[
\left( f | \one_{[a,b]}(H_{\mu})f \right)=\int_0^{\infty}\one_{[a,b]}(\kappa^2)\left|(\cF_{\mu}f)(\kappa)\right|^2\d\kappa
\]
for any $f\in\cC_c^{\infty}(\R_+)$. We may thus identify
\[
\one_{[a,b]}(H_{\mu})=\cF_{\mu}\one_{[a,b]}(Q^2)\cF_{\mu}^*
\]
where $Q$ is the self-adjoint position operator, and in consequence,
\[ 
\cF_{\mu}H_{\mu}\cF_{\mu}^{-1}=Q^2.
\]
\end{theorem}
Note that in particular, $\cF_{\mu}$ is a unitary involution. For any $\gamma\in\C$, one gets
\[
\cF_{\mu}H_{\mu}^{\gamma}\cF_{\mu}^{-1}=Q^{2\gamma},
\]
and in the sense above, the integral kernel of $H_{\mu}^{\gamma}$ is
\[
H_{\mu}^{\gamma}(x,y)=\int_{0}^{\infty} \kappa^{2\gamma}\sqrt{xy}J_{\mu}(x \kappa) J_{\mu}( y\kappa)\kappa\d \kappa,
\]
which can be expressed in terms of the Weber-Schafheitlin integral with $\mu=\nu$.

We quote also the following result concerning the wave operators for $H_{\mu}$, assuming now $\mu\in\R$:
\begin{theorem}
For $\mu,\nu>1$, the M\o ller wave operators $\Omega^{\pm}_{\mu,\nu}$ associated to $H_{\mu}$, $H_{\nu}$ exist and
\[
\Omega^{\pm}_{\mu,\nu}:= \lim_{t\to\pm{\infty}}\e^{itH_{\mu}}\e^{-itH_{\nu}}=\e^{\pm\i(\mu-\nu)\pi/2}\cF_{\mu}\cF_{\nu}.
\]
\end{theorem}
In particular, the integral kernel of the operator $\Omega^{\pm}_{\mu,\nu}H_{\nu}^{\gamma}=H_{\mu}^{\gamma}\Omega^{\pm}_{\mu,\nu}$ may be useful in calculations. We have by the above considerations:
\[
\Omega^{\pm}_{\mu,\nu}H^{\gamma}_{\nu}=\e^{\pm\i(\mu-\nu)\pi/2}\cF_{\mu}Q^{2\gamma}\cF_{\nu}
\]
and its integral kernel is equal to
\[
\Omega^{\pm}_{\mu,\nu}H_{\nu}^{\gamma}(x,y)=\e^{\pm\i(\mu-\nu)\pi/2}\int_{0}^{\infty} \kappa^{2\gamma}\sqrt{xy}J_{\mu}(x \kappa) J_{\nu}( y\kappa)\kappa\d \kappa=\e^{\pm\i(\mu-\nu)\pi/2}\sqrt{\frac{x}{y}}y^{-2\gamma-1}\int_{0}^{\infty} \kappa^{2\gamma+1}J_{\mu}\left(\frac{x}{y} \kappa\right) J_{\nu}\left( \kappa\right)\d \kappa,
\]
where the last integral is of Weber-Schafheitlin type with exponent $2\gamma+1$ and argument $\frac{x}{y}$.

\section{The Weber-Schafheitlin integral with $\Re\rho<1$}\label{watsonowy}

Proceeding as in \cite{DF}, we quote the following classic result \cite{W} for the integral involving the modified Bessel function of the first and second kind, $I_{\mu}$ and $K_{\mu}$:

\begin{lemma}\label{calka1} For $\Re z>0$, $\left|z\right|>1$, \ $\Re(\nu+\rho+1)>\left|\Re\mu\right|$, one has
\[
\int_0^\infty \kappa^{\rho}K_{\mu}(z\kappa)I_{\nu}(\kappa)\d\kappa=\Gamma({\textstyle \frac{1+\rho+\mu+\nu}{2}})\Gamma({\textstyle \frac{1+\rho-\mu+\nu}{2}}) 2^{\rho-1}z^{-1-\rho-\nu} \ _{2}F_1^{\rm I}\left({\textstyle\frac{1+\rho+\mu+\nu}{2}},{\textstyle\frac{1+\rho-\mu+\nu}{2}};\nu+1;z^{-2}\right).
\]
\end{lemma}
\begin{proof}
The conditions for $\mu, \nu, \rho$, for the convergence of the integral, are established using asymptotic series for the Bessel functions of the corresponding type.

For the derivation of the integral, we first rescale the variable $\kappa$ by the factor $z^{-1}$, expand $I_{\nu}(\kappa/z)$ as a power series and then use, assuming $\left|z\right|>1$,
\[
\int_0^{\infty} K_{\mu}(\kappa)\kappa^{\beta-1}=\int_0^{\infty} \int_0^{\infty}\e^{-u \kappa}(u^2-1)^{\mu-\frac{1}{2}} \kappa^{\beta-1}\d u \d \kappa=2^{\beta-2}\Gamma\left({\textstyle \frac{\beta-\mu}{2}}\right)\Gamma\left({\textstyle \frac{\beta+\mu}{2}}\right),
\]
where we have substituted for $K_{\mu}$ the corresponding integral representation, and computed the obtained expression, integrating first with respect to $\kappa$. It remains to compare the obtained series with the hypergeometric $_{2}F_1$ function on the RHS.\qed
\end{proof}

We recall the relation between the Bessel function of the first kind $J_{\mu}$ and $I_{\nu}$:
\[
I_{\nu}(z)=\i^{-\nu}J_{\nu}(\i z).
\]

Using 
\[\int_0^\infty \kappa^{\rho}K_{\mu}(z\kappa)I_{\nu}(\kappa)\d\kappa=z^{-\rho-1}\int_0^\infty \kappa^{\rho}K_{\mu}(\kappa)I_{\nu}(\kappa/z)\d\kappa,
\] 
\[
\int_0^\infty \kappa^{\rho}K_{\mu}(z\kappa)J_{\nu}(\kappa)\d\kappa=z^{-\rho-1}\int_0^\infty \kappa^{\rho}K_{\mu}(\kappa)J_{\nu}(\kappa/z)\d\kappa,
\]
we get as a straightforward corollary of Lemma \ref{calka1}:
\begin{corollary}\label{corrolaryk}For $\Re z>0$, \ $\Re(\nu+\rho+1)>\left|\Re\mu\right|$,
\begin{equation}\label{eq:besselsecondia}
\int_0^\infty \kappa^{\rho}K_{\mu}(z\kappa)J_{\nu}(\kappa)\d\kappa=\Gamma({\textstyle \frac{1+\rho+\mu+\nu}{2}})\Gamma({\textstyle \frac{1+\rho-\mu+\nu}{2}}) 2^{\rho-1}z^{-1-\rho-\nu} \ _{2}F_1^{\rm I}\left({\textstyle\frac{1+\rho+\mu+\nu}{2}},{\textstyle\frac{1+\rho-\mu+\nu}{2}};\nu+1;-z^{-2}\right).
\end{equation}
\end{corollary}
Note that the superfluous assumption $\left|z\right|>1$ has been eliminated by analytic continuation with respect to $z$, using the analyticity of the Gauss hypergeometric function $_{2}F_1$ on $\C\setminus\left[\right.1,\infty\left[\right.$.

We denote the Hankel function of the first and second kind, respectively --- $H^{+}$, $H^{-}$, and recall that they are related to $K_{\mu}$ as follows, for any $y\in\C$: 
\[
H^{+}_{\mu}(y)=\frac{2}{\i\pi} \ \e^{\frac{-\i\pi\mu}{2}}K_{\mu}(-\i y),
\]
\[
H^{-}_{\mu}(y)=-\frac{2}{\i\pi} \ \e^{\frac{\i\pi\mu}{2}}K_{\mu}(\i y).
\]
It follows that the integrals
\[
\int_0^\infty\kappa^{\rho}H^{\pm}_{\mu}(x\kappa)J_{\nu}(\kappa)\d\kappa 
\]
for $x\in\R_+$ are both the limiting case of $\int_0^\infty\kappa^{\rho}K_{\mu}(z\kappa)J_{\nu}(\kappa)\d\kappa$ for purely imaginary $z$. We use the results obtained in Corollary \ref{corrolaryk} for $\Re z>0$, setting first $z=\mp\i(x\pm\i\varepsilon)$, which gives:
\begin{eqnarray*}
\int_0^\infty\kappa^{\rho}H^{\pm}_{\mu}(x\kappa)J_{\nu}(\kappa)\d\kappa=\lim_{\varepsilon\searrow0}\int_0^\infty\kappa^{\rho}H^{\pm}_{\mu}(x\pm\i \varepsilon\kappa)J_{\nu}(\kappa)\d\kappa \\ = \ \frac{2}{\i\pi} \ \e^{\frac{\mp\i\pi\mu}{2}}\lim_{\varepsilon\searrow0}\int_0^\infty\kappa^{\rho}K_{\mu}(\mp\i(x\pm\i\varepsilon)\kappa)J_{\nu}(\kappa)\d\kappa \\ = \ \pm \frac{2^{\rho}}{\i\pi} \e^{\pm\i\pi\frac{1+\rho-\mu+\nu}{2}}\Gamma({\textstyle \frac{1+\rho+\mu+\nu}{2}})\Gamma({\textstyle \frac{1+\rho-\mu+\nu}{2}})\lim_{\varepsilon\searrow0}f(x\pm\i\varepsilon),
\end{eqnarray*}
where
\[
f(x\pm\i\varepsilon):=(x\pm\i\varepsilon)^{-1-\rho-\nu} \ _{2}F_1^{\rm I}\left({\textstyle\frac{1+\rho+\mu+\nu}{2}},{\textstyle\frac{1+\rho-\mu+\nu}{2}};\nu+1;(x\pm\i\varepsilon)^{-2}\right).
\]

Since we have also the relation $J_{\mu}=\frac{1}{2}\left(H^{+}_{\mu}+H^{-}_{\mu}\right)$ \ \cite{W}, it follows that 
\begin{equation}\label{eq:wsprzezf}
\int_0^\infty\kappa^{\rho}J_{\mu}(x\kappa)J_{\nu}(\kappa)\d\kappa=\frac{2^{\rho-1}}{\i\pi}\Gamma({\textstyle \frac{1+\rho+\mu+\nu}{2}})\Gamma({\textstyle \frac{1+\rho-\mu+\nu}{2}})\left[\e^{\i\pi\frac{1+\rho-\mu+\nu}{2}}\lim_{\varepsilon\searrow0}f(x+\i\varepsilon)-\e^{-\i\pi\frac{1+\rho-\mu+\nu}{2}}\lim_{\varepsilon\searrow0}f(x-\i\varepsilon)\right].
\end{equation}
Therefore, in order to derive the Weber-Schafheitlin integral, as well as the integrals involving $H^{\pm}_{\mu}$, it is enough to examine the limit $\lim_{\varepsilon\searrow0}f(x\pm\i\varepsilon)$. Note that $f(z)$ depends on the parameters $\mu, \nu, \rho$ and it will follow that it is convenient to treat the cases $\Re\rho<0$ and $\Re\rho>0$ separately. 


\begin{proposition}\label{wynikfunkcja}For any $\mu,\nu\in\C$ and $\Re\rho<0$ satisfying $\Re\left(\rho+\nu+1\right)>\left|\mu\right|$, and $x\in\R_+$, the integral $\int_0^\infty\kappa^{\rho}J_{\mu}(x\kappa)J_{\nu}(\kappa)\d\kappa$ is equal to:
\[
2^{\rho}\frac{\Gamma(\frac{1+\rho+\mu+\nu}{2})}{\Gamma(\frac{1-\rho-\mu+\nu}{2})} x^{\mu} \ _{2}F_1^{\rm I}\left({\textstyle\frac{1+\rho+\mu+\nu}{2}},{\textstyle\frac{1+\rho+\mu-\nu}{2}};\mu+1;x^{-2}\right)
\]
for $x<1$, and
\[
2^{\rho}\frac{\Gamma(\frac{1+\rho+\mu+\nu}{2})}{\Gamma(\frac{1-\rho+\mu-\nu}{2})} x^{-1-\rho-\nu} \ _{2}F_1^{\rm I}\left({\textstyle\frac{1+\rho+\mu+\nu}{2}},{\textstyle\frac{1+\rho-\mu+\nu}{2}};\nu+1;x^{-2}\right)
\]
for $x>1$.
\end{proposition}
\begin{proof}
Consider the factor $_{2}F_1^{\rm I}\left({\textstyle\frac{1+\rho+\mu+\nu}{2}},{\textstyle\frac{1+\rho-\mu+\nu}{2}};\nu+1;(x\pm\i\varepsilon)^{-2}\right)$, appearing in the definition of $f(x+\i\varepsilon)$.
Denoting $a:=\frac{1+\rho+\mu+\nu}{2}$, $b:=\frac{1+\rho-\mu+\nu}{2}$ and $c:=\nu+1$, we check that $\Re (c-a-b)=-\Re\rho<0$, which ensures the limit $\varepsilon\searrow0$ exists.

For $x>1$, the argument of the $_{2}F_1$ function has real part smaller than $1$, thus by analyticity it follows that the limits with $+\i\varepsilon$ and $-\i\varepsilon$ coincide. Obtaining the desired expression is then just a matter of rewriting the phase factors in terms of $\Gamma$ functions, using
\[
\frac{1}{2\i}\left(\e^{\i\pi\frac{1+\rho-\mu+\nu}{2}}-\e^{-\i\pi\frac{1+\rho-\mu+\nu}{2}}\right)=\sin\left({\textstyle\pi\frac{1+\rho-\mu+\nu}{2}}\right)=\frac{\pi}{\Gamma\left({\textstyle\frac{1+\rho+\nu-\mu}{2}}\right)\Gamma\left({\textstyle\frac{1-\rho-\nu+\mu}{2}}\right)}.
\]

For $x<1$, we can use the result above with $\mu$ and $\nu$ interchanged, thanks to
\[
\int_0^\infty \kappa^{\rho}J_{\mu}(x\kappa)J_{\nu}(\kappa)\d\kappa=x^{-\rho-1}\int_0^\infty \kappa^{\rho}J_{\mu}(\kappa)J_{\nu}(\kappa/x)\d\kappa.\] 
\qed  
\end{proof}

\section{The case $\Re\rho>0$} \label{sekcjahankel}

Assuming $\Re\rho>0$, we examine the limit
\[
\lim_{\varepsilon\searrow0}f(x\pm\i\varepsilon)=\lim_{\varepsilon\searrow0} \ \ (x\pm\i\varepsilon)^{-1-\rho-\nu} \ _{2}F_1^{\rm I}\left({\textstyle\frac{1+\rho+\mu+\nu}{2}},{\textstyle\frac{1+\rho-\mu+\nu}{2}};\nu+1;(x\pm\i\varepsilon)^{-2}\right).
\]
Denoting $a:=\frac{1+\rho+\mu+\nu}{2}$, $b:=\frac{1+\rho-\mu+\nu}{2}$ and $c:=\nu+1$, we check that $\Re (c-a-b)=-\Re\rho>0$, which implies a singular behaviour of the $_{2}F_1^{\rm I}$ factor at $x=1$ in the limit $\varepsilon\to 0$. We therefore use the following symmetry of the $_{2}F_1^{\rm I}$ function:
\begin{eqnarray*}
_{2}F_1^{\rm I}\left(a,b;c;z^{-2}\right)= (1-z^{-2})^{c-a-b} \ _{2}F_1^{\rm I}\left(c-a,c-b;c;z^{-2}\right) \\ \ = (1-z^{-2})^{-\rho} \ _{2}F_1^{\rm I}\left({\textstyle\frac{1-\rho+\mu+\nu}{2}},{\textstyle\frac{1-\rho-\mu+\nu}{2}};\nu+1;z^{-2}\right),
\end{eqnarray*}
in order to isolate the whole singular factor $(1-z^{-2})^{-\rho}$, the second term on the RHS being well defined for all $z\in\C$. Consequently,
\begin{eqnarray*}
f(x\pm\i\varepsilon)=(x\pm\i\varepsilon)^{-1-\rho-\nu}(1-(x\pm\i\varepsilon)^{-2})^{-\rho} \ _{2}F_1^{\rm I}\left({\textstyle\frac{1-\rho+\mu+\nu}{2}},{\textstyle\frac{1-\rho-\mu+\nu}{2}};\nu+1;(x\pm\i\varepsilon)^{-2}\right)
\\ = (x\pm\i\varepsilon)^{-1+\rho-\nu}(x+1\pm\i\varepsilon)^{-\rho}(x-1\pm\i\varepsilon)^{-\rho} \ _{2}F_1^{\rm I}\left({\textstyle\frac{1-\rho+\mu+\nu}{2}},{\textstyle\frac{1-\rho-\mu+\nu}{2}};\nu+1;(x\pm\i\varepsilon)^{-2}\right).
\end{eqnarray*}
Because of the singular part (i.e. $(x-1\pm\i\varepsilon)^{-\rho}$), the limit $\varepsilon\searrow0$ requires a distributional approach. 
Let us concentrate now on the first factors of the above expression. We will need the following:
\begin{definition}
We define $\left(x-1\pm\i0 \right)^{\lambda}:=\lim_{\varepsilon\searrow0}(x-1\pm\i\varepsilon)^{\lambda}$, as distributions on $\R_+$.
\end{definition}
\begin{remark}\label{remark1}
The distributions $(x-1\pm\i0)^{\lambda}$ are well defined and can be represented in a more explicit way as follows (see \cite{H} for properties of the analogously defined distributions $(x\pm\i0)^\lambda$ on $\R$):
\begin{innerlist}
\item for $\Re\lambda>-1$, we have the equality between locally integrable functions 
\[
(x-1\pm\i0)^{\lambda}= \e^{\pm\i\lambda\pi}(x-1)^{\lambda}_- +(x-1)^{\lambda}_+,  
\]
where:
\[
(x-1)^{\lambda}_{-}:=\begin{cases}|x-1|^{\lambda},& x<1\\ 0, & x\geq1 \end{cases},
\]
\[
(x-1)^{\lambda}_{+}:=\begin{cases}0,& x<1\\ (x-1)^{\lambda}, & x\geq1 \end{cases};
\]
\item for $\Re\lambda>-k$ and $\lambda\notin\Z$, where $k\in\N$, we have the equality
\[
(x-1\pm\i0)^{\lambda}= \e^{\pm\i\lambda\pi}(x-1)^{\lambda}_- +(x-1)^{\lambda}_+,  
\]
where $(x-1)^{\lambda}_{\pm}$ are now distributions defined by their action on an arbitrary test function $\phi\in \cc^{\infty}_c(\R_+)$ (or equivalently, as the analytic continuation of the distributions defined in the preceding case):
\[
\left\langle (x-1)^{\lambda}_{-} , \phi \right\rangle:= \int_0^1 (1-x)^{\lambda+k}\phi^{(k)}(x)/\left((\lambda+1)\ldots(\lambda+k)\right) \d x ,
\]
\[
\left\langle (x-1)^{\lambda}_{+} , \phi \right\rangle:=(-1)^k \int_1^\infty (x-1)^{\lambda+k}\phi^{(k)}(x)/\left((\lambda+1)\ldots(\lambda+k)\right) \d x .
\]
\item for $\lambda=-k$, where $k\in\N$, we have
\[
(x-1\pm\i0)^{-k}:=(x-1)^{-k}_- +(-1)^{k}(x-1)^{-k}_+ \pm (-1)^{k}\i\pi\frac{\delta^{(k-1)}(x-1)}{(k-1)!},
\]
where
\[
\left\langle (x-1)^{-k}_{-} , \phi \right\rangle:=(-1)^{k-1}\int_0^1 \log(1-x)\frac{\phi^{(k)}(x)}{(k-1)!}\d x + (-1)^{k-1} \frac{\phi^{(k-1)}(1)\left(\sum_{j=1}^{k-1}j^{-1}\right)}{(k-1)!},
\]
\[
\left\langle (x-1)^{-k}_{+} , \phi \right\rangle:=-\int_1^{\infty} \log(x-1)\frac{\phi^{(k)}(x)}{(k-1)!}\d x + \frac{\phi^{(k-1)}(1)\left(\sum_{j=1}^{k-1}j^{-1}\right)}{(k-1)!}.
\]
\end{innerlist}
\end{remark}
\begin{lemma}
Let $\lambda\in\C$ and let $g(z)$ be a function holomorphic in the neighborhood of the halfline $\R_+$. Then
\[
\lim_{\varepsilon\searrow0} \left[g(x+\i\varepsilon)(x-1\pm\i\varepsilon)^{\lambda}\right]=g(x)(x-1\pm\i 0)^{\lambda},
\]
in the sense of distributions on $\R_+$.\label{l:lemat1}
\end{lemma}
\proof It is clear that 
\begin{equation}\label{eq:limit1}     
\lim_{\varepsilon\searrow0} \left[g(x+\i\varepsilon)(x-1\pm\i\varepsilon)^{\lambda}\right]=g(x)(x-1\pm\i 0)^{\lambda},
\end{equation}
for $\Re\lambda>-1$, since the RHS is then a regular distribution and the convergence of each of the factors as functions is uniform on every compact subset of $\R_+$. Assuming $\lambda\neq 0$, we differentiate (\ref{eq:limit1}) and obtain
\[
\lim_{\varepsilon\searrow0} \left[g'(x+\i\varepsilon)(x-1\pm\i\varepsilon)^{\lambda} +\lambda g(x+\i\varepsilon)(x-1\pm\i\varepsilon)^{\lambda-1}\right]
=g'(x)(x-1\pm\i 0)^{\lambda}
+\lambda g(x)(x-1\pm\i0)^{\lambda-1},
\]
where we have used the fact that $\frac{\d}{\d x}(x-1\pm\i 0)^{\lambda}=\lambda(x-1\pm\i 0)^{\lambda-1}$ \ \cite{H}. We then use (\ref{eq:limit1}) again to substract the first term of both sides. We have thus proved (\ref{eq:limit1}) for $\Re\lambda>-2, \lambda\neq -1$. 

We consider the case $\lambda=-1$ separately, and differentiate instead the equality between the locally integrable functions:
\[
\lim_{\varepsilon\searrow0} \left[g(x+\i\varepsilon)\log(x-1\pm\i\varepsilon)\right]=g(x)\log(x-1\pm\i0),
\]
where $\log(x-1\pm\i0):=\lim_{\varepsilon\searrow0}\log(x-1\pm\i\varepsilon)=\log\left|x-1\right|\pm\i\pi\theta(x-1)$, and its distributional derivative is $(x-1\pm\i0)^{-1}$ (this can be seen by setting first $g(x)\equiv1$ in the above equality and differentiating).

By induction, we prove (\ref{eq:limit1}) for the remaining values of $\lambda\in\C$.
\qed

Recalling our expression for $f(x+\i\varepsilon)$, we have
\[
f(x+\i\varepsilon)=(x+\i\varepsilon)^{-1+\rho-\nu}(1+x+\i\varepsilon)^{-\rho}(x-1+\i\varepsilon)^{-\rho}q(x+\i\varepsilon),
\] 
where 
\[
q(x+\i\varepsilon):= \ _{2}F_1^{\rm I}\left({\textstyle \frac{1-\rho+\mu+\nu}{2}},{\textstyle\frac{1-\rho-\mu+\nu}{2}};\nu+1;(x+\i\varepsilon)^{-2}\right),
\]
which converges uniformly on every compact subset of $\R_+$ to
\[
q(x+\i0):= \ _{2}F_1^{\rm I}\left({\textstyle \frac{1-\rho+\mu+\nu}{2}},{\textstyle\frac{1-\rho-\mu+\nu}{2}};\nu+1;x^{-2}\right),
\]
considered as a single-valued function. Note that the Gauss hypergeometric function $_{2}F_1$ has a branch cut $\left[\right.1,\infty\left[\right.$. Here, the limit is taken by approaching the halfline from below ($\Im(x+\i\varepsilon)^{-2}<0$), and we have denoted $_{2}F_1(a,b;c;x)=\lim_{\epsilon\searrow0} \ _{2}F_1(a,b;c;x-\i\epsilon)$ on the branch cut. On the other hand, the limit $\lim_{\varepsilon\searrow0} f(x-\i\varepsilon)$ corresponds to approaching the halfline from above in the argument of the $_{2}F_1$ factor, therefore
\[
q(x-\i0):=\lim_{\varepsilon\searrow0}q(x-\i\varepsilon)
\]
is not equal to $q(x+\i0)$. 

\begin{remark}\label{remark2}
The function $x\mapsto q(x\pm\i0)$ is not differentiable, and therefore the meaning of the product $(x-1\pm\i0)^{-\rho} q(x\pm\i0)$ is unclear. To prove that it exists despite this apparent problem, we show that $q(x\pm\i0)$ can be written as
\begin{equation}\label{eq:rozklad}
q(x\pm\i0)=h_1(x)+(x-1\pm\i0)^{\rho}h_2^{\pm}(x),
\end{equation}
where $h_1(x)$ is smooth and both $h_2^{+}(x)$, $h_2^{-}(x)$ belong to $L_1^{{\rm loc}}(\R_+)$. Then, the equality
\[
(x-1\pm\i0)^{-\rho} \ q(x\pm\i0):=(x-1\pm\i0)^{-\rho}h_1(x)+h_2^{\pm}(x)
\]
defines the desired product well, being the sum of a distribution multiplied by a smooth function and of a locally integrable function.\end{remark}
\begin{proof}For $\rho\notin\Z$, the decomposition (\ref{eq:rozklad}) is possible due to the following formula, holding for $z\in\C$ (\cite{BS}, eq. (B.9)):
\begin{eqnarray}\label{eq:decomposition}
_{2}F_1^{\rm I}\left({\textstyle\frac{1-\rho+\mu+\nu}{2}},{\textstyle\frac{1-\rho-\mu+\nu}{2}};\nu+1;z\right)= \frac{\pi}{\sin\pi\rho} \bigg\{ \frac{1}{\Gamma({\textstyle \frac{1+\rho+\mu+\nu}{2}})\Gamma({\textstyle \frac{1+\rho-\mu+\nu}{2}})} \ _{2}F_1^{\rm I}\left({\textstyle \frac{1-\rho+\mu+\nu}{2}},{\textstyle\frac{1-\rho-\mu+\nu}{2}};-\rho+1;1-z\right) \nonumber \\ - \ \frac{1}{\Gamma({\textstyle \frac{1-\rho+\mu+\nu}{2}})\Gamma({\textstyle \frac{1-\rho-\mu+\nu}{2}})} (1-z)^{\rho}\ _{2}F_1^{\rm I}\left({\textstyle \frac{1+\rho+\mu+\nu}{2}},{\textstyle\frac{1+\rho-\mu+\nu}{2}};\rho+1;1-z\right)\bigg\}.
\end{eqnarray}
We take $z=(x\pm\i\varepsilon)^{-2}$ and pass to the limit $\varepsilon\searrow0$ with both the expressions. Note that the only difference between the limit with $+\i\varepsilon$ and $-\i\varepsilon$ appears in the factor  $\lim_{\varepsilon\searrow0}(1-(x\pm\i\varepsilon)^{-2})^\rho=x^{-2\rho}(1+x)^{\rho}(x-1\pm\i0)^{\rho}$, since the $_{2}F_1$ factors on the RHS are analytic on the required domain. It is clear that (\ref{eq:rozklad}) holds, with $h_2^{+}(x)=h_2^{-}(x)$ in particular. 

For $\rho\in\Z$, we have instead (\cite{BS}, eq. (B.10)):
\begin{eqnarray*}
_{2}F_1^{\rm I}\left({\textstyle\frac{1-\rho+\mu+\nu}{2}},{\textstyle\frac{1-\rho-\mu+\nu}{2}};\nu+1;z\right)= \\  \frac{1}{\Gamma({\textstyle \frac{1+\rho+\mu+\nu}{2}})\Gamma({\textstyle \frac{1+\rho-\mu+\nu}{2}})}\sum_{k=0}^{\rho-1}\frac{(-1)^k (\rho-k-1)!\left({\textstyle \frac{1-\rho+\mu+\nu}{2}}\right)_k \left({\textstyle \frac{1-\rho-\mu+\nu}{2}}\right)_k}{k!}(1-z)^k + \\ \ \frac{(-1)^\rho}{\Gamma({\textstyle \frac{1-\rho+\mu+\nu}{2}})\Gamma({\textstyle \frac{1-\rho-\mu+\nu}{2}})}(1-z)^{\rho}\sum_{k=0}^{\infty}\frac{\left({\textstyle \frac{1+\rho+\mu+\nu}{2}}\right)_k \left({\textstyle \frac{1+\rho-\mu+\nu}{2}}\right)_k}{(\rho+k)!k!}\left\{\psi(k+1)+\psi(\rho+k+1)\right. \\ \left.-\psi({\textstyle \frac{1+\rho+\mu+\nu}{2}}+k)-\psi({\textstyle \frac{1+\rho-\mu+\nu}{2}}+k)-\log(1-z)\right\}(1-z)^k,
\end{eqnarray*}
for ${\textstyle\frac{1-\rho\pm\mu+\nu}{2}}\notin\Z$, where $\psi(y)=\Gamma'(y)/\Gamma(y)$. 

As in the preceding case, we take $z=(x\pm\i\varepsilon)^{-2}$ and pass to the limit. We obtain a similar decomposition. Note, however, that the limits $\log(1-x\mp\i0)$ differ, and thus $h_2^{+}(x)\neq h_2^{-}(x)$  

 The cases where at least one of the parameters ${\textstyle\frac{1-\rho\pm\mu+\nu}{2}}$ is an integer are treated the same way, only the functions $h_1(x)$ $h_2^{\pm}(x)$ being then different (\cite{BS}, eq. (B.11), (B.12)).\qed
\end{proof}  

\begin{proposition}\label{eq:wynik1} In the sense of distributions on $\R_+$,
\[
\lim_{\varepsilon\searrow0}f(x\pm\i\varepsilon)=x^{-1+\rho-\nu}(x+1)^{-\rho}\left[(x-1\pm\i0)^{-\rho}h_1(x)+h_2^{\pm}(x)\right],
\]
where the functions $h_1(x)$ and $h_2^{\pm}(x)$ are defined in Remark \ref{remark2}.
\end{proposition}
\begin{proof}
We have to prove that as $\varepsilon\searrow0$,
\[
(x-1\pm\i\varepsilon)^{-\rho}(x\pm\i\varepsilon)^{-1+\rho-\nu}(x+1\pm\i\varepsilon)^{-\rho}h_1(x\pm\i\varepsilon) \longrightarrow (x-1\pm\i0)^{-\rho}x^{-1+\rho-\nu}(x+1)^{-\rho}h_1(x)
\]
and
\[
(x\pm\i\varepsilon)^{-1+\rho-\nu}(x+1\pm\i\varepsilon)^{-\rho}(x-1\pm\i\varepsilon)^{-\rho}(x-1\pm\i\varepsilon)^{\rho}h_2(x\pm\i\varepsilon) \longrightarrow x^{-1+\rho-\nu}(1+x)^{-\rho}h_2^{\pm}(x).
\]
The first limit is a consequence of Lemma \ref{l:lemat1}. The second one is clearly true, since the convergence of the corresponding functions is uniform on each compact subset of $\R_+$.
\qed
\end{proof}

\begin{remark}
The case when $\rho\in\Z$ and at least one of the numbers ${\textstyle\frac{1-\rho\pm\mu+\nu}{2}}$ is an integer are treated similarly. One can deduce from the expansions given in \cite{BS} (eq. (B.11), (B.12)), the explicit expressions for the functions $h_1(x)$ and $h_2^{\pm}(x)$, following step by step the proof of Remark \ref{remark2} in the degenerate case. 
\end{remark}

\begin{corollary}Using Equation \ref{eq:wsprzezf} and Proposition \ref{eq:wynik1}, we get the results gathered in Proposition \ref{wynik2}. The formulae in Proposition \ref{wynik2} are written in terms of the distributions $(x-1)_{\pm}^{-\rho}$ rather than $(x-1\pm\i0)^{-\rho}$, using the relations listed in Remark \ref{remark1}.
\end{corollary}

We end up commenting on some special cases, involving much simplier expressions.

\begin{remark}
An explicit formula in the special case $\rho=1$ has been derived in \cite{KR}. It can be recovered from our general expression, by substituting the (well-known) equality
\[
(x-1\pm\i0)^{-1}=\Pv \left(\frac{1}{x-1}\right)\mp\i\pi\delta (x-1),
\]
where $\Pv$ denotes the Cauchy principal value. Furthermore, $S_{\mu,\nu,1}(x)\equiv1$ by definition, and it remains to use (\ref{eq:decomposition}) back again to get $ _{2}F_1$ functions with argument $x^{\pm2}$ instead of $1-x^{\pm2}$.
\end{remark}

\begin{remark}
Much simplier formulae can be derived in the special case $1-\rho+\mu\pm\nu=0$, since the Gauss hypergeometric function with a parameter set to zero is trivial, i.e. $_{2}F_1(0,\ldots;\cdot)\equiv1$.
\end{remark}

Note that the above remarks hold for the integral involving $H_{\mu}^{\pm}$ instead of $J_{\mu}$ as well.


\ \\
\hrule
\small
\textsc{Micha{\l} Wrochna},  \emph{Research Training Group ``Mathematical Structures in Modern Quantum Physics'',
\\ Mathematisches Institut, Universit\"at G\"ottingen, Bunsenstr. 3-5, D - 37073 G\"ottingen, Germany
\vspace{0.4\baselineskip}
\\ e-mail: \upshape{wrochna@uni-math.gwdg.de}}


\begin{thebibliography}{MK}

 
\bibitem[BDG]{BDG} Bruneau L., Derezi\'{n}ski J., Georgescu V.: \emph{Homogeneous operators on halfline}, arXiv:0911.5569v1, 2009
 
\bibitem[BS]{BS} Becken W., Schmelcher P.: \emph{The analytic continuation of the Gaussian hypergeometric function for arbitrary parameters}, Journal of Computational and Applied Mathematics 126, p. 449-478, 2000

\bibitem[DF]{DF} Dixon A. L., Ferrar W. L.: \emph{Infinite integrals in the theory of Bessel Functions}, Quarterly Journal of Mathematics 126, p. 122-145, 1930


\bibitem[H]{H} H\"{o}rmander I.: \emph{The analysis of linear partial differential operators , volume 1}, 2nd edition, Springer, 1990

\bibitem[HN]{HN} Hongo K., Naqvi Q.A.: \emph{Diffraction of electromagnetic wave by disk and circular hole in a perfectly conducting plane}, Progress in Electromagnetics Research, 113–150, 2007

\bibitem[KeR]{KeR} Keating J.P.1; Robbins J.M: \emph{Force and impulse from an Aharonov-Bohm flux line}, Journal of Physics A: Mathematical and General, vol. 34 no. 4, 2001 

\bibitem[KR]{KR} Kellendonk J., Richard S.: \emph{Weber-Schafheitlin type integrals with exponent 1}, Integral Transforms and Special Functions 20 no. 2, 2009

\bibitem[KR]{KR2} Kellendonk J., Richard S.: \emph{New formulae for the Aharonov-Bohm wave operators}, arXiv:0811.3963, 2008

\bibitem[M]{M} Miroshin, R.N.: \emph{An asymptotic series for the Weber-Schafheitlin integral}, Math. Notes 70, no. 5-6 p. 682–687, 2001

\bibitem[SS]{SS} Suzuki H., Sato H.-T.: \emph{On Bogoliubov transformation of scalar wave functions in de Sitter space}, Mod. Phys. Lett. A9 3673-3684, 1994

\bibitem[SW]{SW} Salamon N.J., Walter G.G.: \emph{Limits of Lipschitz-Hankel Integrals}, J. Inst. Maths Applics 24, 237-254, 1979

\bibitem[W]{W} Watson G.N.: \emph{A treatise on the theory of Bessel functions}, 2nd edition, Cambridge University Press, 1966

\end{thebibliography}
\end{document}